%% file: gershgorin_abstract_v7.tex
\documentclass[a4paper]{amsart}
 
\usepackage{amsmath,amssymb,amsthm,amsfonts,enumerate,url,fancyhdr,mathbbol,mathrsfs,color}
\usepackage{hyperref}
\usepackage[utf8]{inputenc}

\usepackage{color}

\usepackage[english]{babel}
\newtheorem{theo}{Theorem}
\newtheorem{lema}[theo]{Lemma}

\newtheorem{coro}[theo]{Corollary}

\newtheorem{defi}[theo]{Definition}

\newtheorem{assum}[theo]{Assumptions}
\newtheorem{rema}[theo]{Remark}

\newtheorem{exa}[theo]{Example}

\numberwithin{equation}{section}
\numberwithin{theo}{section}
\theoremstyle{definition}

\DeclareMathOperator{\Id}{Id}

\newcommand{\coloneqq}{\mathrel{\mathop:}=}

\newcommand{\diag}{{\rm diag}}

\newcommand{\R}{\mathbb{R}}
\newcommand{\C}{\mathbb{C}}

\newcommand{\N}{\mathbb{N}}

\newcommand{\A}{\mathbb{A}}

\title[A new Gershgorin-type result]{A new Gershgorin-type result for the localisation of the spectrum of matrices}

\author{Anna Dall'Acqua}
\address{Anna Dall'Acqua, Institut f\"ur Analysis, Fakultät für Mathematik und Wirtschafts\-wissenschaften, Universit\"at Ulm, 89081 Ulm, Germany}
\email{anna.dallacqua@uni-ulm.de}

\author{Delio Mugnolo}
\address{Delio Mugnolo, Lehrgebiet Analysis, Fakultät für Mathematik und Informatik, Fern\-{U}niversität Hagen, 58084 Hagen, Germany}
\email{delio.mugnolo@fernuni-hagen.de}

\author{Michael Schelling}
\address{Michael Schelling, Institut f\"ur Analysis, Fakultät für Mathematik und Wirtschafts\-wissenschaften, Universit\"at Ulm, 89081 Ulm, Germany}
\email{michael.schelling@uni-ulm.de}

\subjclass[2010]{15A42, 47A10}

\keywords{Estimates of the spectrum of matrices,  Operator matrices of unbounded operators}
\thanks{The second author was partially supported by the Land Baden--W\"urttemberg in the framework of the \emph{Juniorprofessorenprogramm} -- research project on ``Symmetry methods in quantum graphs''.}

\begin{document}

\maketitle

\input{abstract_v7.tex}

\input{introduction_v7.tex}

\input{schur_v7.tex}

\input{setting_v7.tex}

\input{main_results_v7.tex}

\input{other_results_v7.tex}

\input{scalar_matrices_v7.tex}

\bibliographystyle{alpha}

\end{document}

%% file: abstract_v7.tex
\begin{abstract}
	We present a Gershgorin's type result on the localisation of the spectrum of a matrix. 
	Our method is elementary and relies upon the method of Schur complements, 
	furthermore it outperforms the one based on the Cassini ovals of Ostrovski and Brauer. 
	Furthermore, it yields estimates that hold without major differences in the cases of both scalar and operator matrices. 
	Several refinements of known results are obtained.
\end{abstract}

%% file: introduction_v7.tex
\section{Introduction}

Gershgorin proved in~\cite{Ger31} a celebrated estimate for the eigenvalues of a scalar $(n\times n)$ matrix
\begin{equation}\label{eq:gershmatr}
	\A=(A_{ij})_{i,j=1}^n\equiv \begin{pmatrix}
	A_{11} & \ldots & A_{1n}\\
	\vdots & \ddots & \vdots\\
	A_{n1} & \ldots & A_{nn}
	\end{pmatrix}\ .
\end{equation}
His result says that the eigenvalues of $\A$ are contained in the union of the sets
	\[
		\left\{\lambda\in \mathbb C\ \Big\vert\ |\lambda-A_{jj}|\le \hspace{-.35cm} 
		\sum_{\quad i=1, i\ne j}^n \big|A_{ij}\big|\right\},\qquad j=1,\ldots,n\ ,
	\]
which are nowadays called \emph{Gershgorin disks}.
This estimate is rather rough but has nevertheless interesting applications and, 
above all, it can be exploited upon performing only very easy computations. 
Salas has observed in~\cite{Sal99} that Gershgorin's theorem carries over to the case of operator matrices 
i.e., to schemes like that in \eqref{eq:gershmatr} where the entries $A_{ij}$ are not scalars, but rather linear operators. 
For such operators, with the same arguments as in \cite[Thm. 2.7]{Sal99} 
one finds that if namely all $A_{ij}$ are bounded operators (rather than scalars), then
	\[
		\sigma(\A)\subset G_1\cup \ldots \cup G_n,
	\]
where for $j=1, ..,n$
\begin{equation}\label{eq:Gi}
	G_j:=\sigma(A_{jj})\cup \left\{\lambda\in \mathbb C \Big\vert \lambda \not\in \sigma(A_{jj})\hbox{ and }
	\big(\|(\lambda-A_{jj})^{-1}\|\big)^{-1}\le \hspace{-.35cm}\sum_{\quad i=1, i\ne j}^n\big\|A_{ij}\big\| \right\} .
\end{equation}
If the operators $A_{ii}$ on the diagonal are not bounded, but only closed 
-- so that for each $\lambda$ outside the spectrum of $A_{ii}$ the inverse $(\lambda-A_{ii})^{-1}$  is still bounded -- 
then the same arguments work and one can see that the above mentioned result still holds in this more general case. 
However, the case of \emph{off-diagonal} unbounded entries is subtler and is the case we are interested in. 
More precisely, we consider $(n \times n)$ operator matrices with unbounded and closed elements in the diagonal 
while the off-diagonal elements are relatively bounded. 
The main idea to describe the spectrum of such matrices is to use the tool that in linear algebra usually goes under 
the name of \emph{Schur complement}.  
We refer to~\cite{Zha05} for a comprehensive treatment of Schur complements. 

Our main result, which we will present in Section~\ref{sec:main}, 
reads as follows in the special case of $(2 \times 2)$ operator matrices. 
A similar idea has been used in~\cite{Nag89} to obtain a different characterisation.
\begin{theo}[The $(2 \times 2)$ case]\label{theo:2by2}
	Let $X_1,X_2$ be complex Banach spaces and consider the product Banach space $X\coloneqq X_1\times X_2$. 
	Let $A: \mathcal{D}(A) \subset X_1 \to X_1$ and $D: \mathcal{D}(D) \subset X_2 \to X_2$ be closed, 
	$B: \mathcal{D}(B) \subset X_2 \to X_1$ be relatively $D$-bounded and $C: \mathcal{D}(C) \subset X_1 \to X_2$ be relatively $A$-bounded. 
	Consider the operator matrix
	\begin{align*}
		\A=\begin{pmatrix}
		A & B\\ C & D
		\end{pmatrix}: \mathcal{D}(A)\times\mathcal{D}(D) \subset X \rightarrow X\ ,
	\end{align*}
	and assume that $\A$ is closed on $\mathcal{D}(A)\times\mathcal{D}(D)$.\newline
	Then,
		$$\sigma(\A) \subset \sigma(A)\cup \sigma(D) \cup \{ \lambda \in \C\setminus (\sigma(A)\cup \sigma(D)): 
		\mathcal{R}_{21}(\lambda) \geq 1\} =: S_{21}(\A)\, , $$
	where
		$$\mathcal{R}_{21}(\lambda) : = \| B (\lambda -D)^{-1} C (\lambda -A)^{-1}\| \, .$$
\end{theo}
Throughout this article we call $S_{21}(\A)$ a \emph{Schur set}.

We are going to prove Theorem~\ref{theo:2by2} in Section~\ref{sec:2times2}. 
The extension to the case of general $(n\times n)$ operator matrices is less trivial than one may imagine.
To this aim, we are going to treat a generic $(n \times n)$ operator matrix as a $(2 \times 2)$ block operator matrix, 
with the upper-left block being a $\bigl((n-1)\times (n-1)\bigr)$ matrix and the lower-right block a $(1\times 1)$ matrix.

Further estimates of the eigenvalues of a matrix are known. 
In particular, it is known since~\cite{Ost37,Bra47} that the spectrum of a scalar matrix $\A=(A_{ij})_{i,j=1}^n$ 
is contained in the union of the so-called \emph{Cassini ovals}
	\[
		\left\{\lambda\in \mathbb C \ \Big\vert\ |\lambda-A_{ii}|\cdot |\lambda-A_{jj}|\le \Big(\hspace{-.35cm}
		\sum_{ \quad k=1, k \ne i}^n |A_{ik}|\Big)\Big(\hspace{-.35cm}\sum_{\quad k=1, k \ne j}^n |A_{jk}|\Big) \right\}\, .
	\] 
This estimate is known to be strictly sharper than Gershgorin's, cf.\ the interesting survey in~\cite{BunWeb12}. 
It, too, can be partially extended to general operator matrices of bounded linear operators.
This has been done in~\cite[\S~5]{HerSch07}. 
So far, Cassini-type inclusions have been proved merely for the \emph{approximate point spectrum} of such operator matrices. 
With our method and under suitable assumptions, we can prove that the whole spectrum of an operator matrix is contained 
in the Cassini ovals, cf.~Theorem \ref{theo:nagel-implies-cassini}.

The paper is organised as follows. 
In Section 2 we study the case $n=2$ and prove Theorem \ref{theo:2by2}. 
The notation for the general case is given in Section 3 while Section 4 contains the main results of the papers: 
the generalisation of Theorem \ref{theo:2by2} to $(n \times n)$ operator matrices (Theorem \ref{theo:nagel-implies-gershgorin}) 
and the fact that the whole spectrum is contained in the Cassini ovals, cf.~Theorem \ref{theo:nagel-implies-cassini}. 
In Section 5 we describe some situations in which our main results hold. 
In Subsection 5.1 we define the \emph{modified Schur sets}: 
these have the advantage that they contain the whole spectrum of the operator matrix under milder assumptions. 
In Subsection 5.2 we present a set of assumptions for the off-diagonal entries of the matrix operator 
which assures that the operator matrix and all its upper-left square blocks are closed. 
In Section 6 we consider the case of scalar matrices and present two examples showing how our own estimate of the spectrum is 
strictly sharper than the one given by the method based on Cassini ovals.

%% file: schur_v7.tex
\section{Schur's Lemma and the $(2 \times 2)$ matrix case}\label{sec:2times2}



For convenience we start by recalling some known facts and definitions. 
Let $X$, $Y$ be complex Banach-spaces. 
An operator $A:\mathcal{D}(A)\subset X\rightarrow X$ is called \emph{closed} 
if its domain $\mathcal{D}(A)$ is a Banach space when endowed with the \emph{graph norm}
	$$ \|x\|_{\mathcal{D}(A)}\coloneqq\|x\|_{X}+\|A x\|_{X}\, , \qquad x\in\mathcal{D}(A) \, .$$
An operator $C: \mathcal{D}(C)\subset X\rightarrow Y$ is called \emph{relatively $A$-bounded} 
if $\mathcal{D}(A) \subset \mathcal{D}(C)$ and there exist $\alpha,\beta \geq 0$ such that
	$$ \| Cx\|_{Y} \leq \alpha \| x \|_{X}+ \beta \| Ax\|_{X} \mbox{ for all }x \in \mathcal{D}(A) \, .$$ 
If in particular $A$ has non-empty resolvent set, then $C$ is relatively $A$-bounded 
if  and only if $C(\lambda-A)^{-1}: X \to Y$ is bounded for one (and thus all) 
$\lambda\not\in\sigma(\A)$, cf.~\cite[Exer.\ III.2.18.1]{EngNag00}.

The following result is a small extension of~\cite[Thm.~2.4]{Nag89}. 

\begin{lema}[Schur's Lemma]\label{lem:schur}
	Let $X_1,X_2$ be complex Banach spaces and consider the product Banach space $X\coloneqq X_1\times X_2$ endowed with the 1-norm. 
	Let $A: \mathcal{D}(A) \subset X_1 \to X_1$ and $D: \mathcal{D}(D) \subset X_2 \to X_2$ be closed, 
	$B: \mathcal{D}(B) \subset X_2 \to X_1$ be relatively $D$-bounded and $C: \mathcal{D}(C) \subset X_1 \to X_2$ be relatively $A$-bounded. 
	Consider the operator matrix 
	\begin{align*}
		\A=\begin{pmatrix}
			A & B\\ C & D
		\end{pmatrix}
		: \mathcal{D}(A)\times\mathcal{D}(D) \subset X \rightarrow X\ ,
	\end{align*}
	and assume that $\A$ is closed on $\mathcal{D}(A)\times\mathcal{D}(D)$.	\\
	For $\lambda\not\in\sigma(D)$ the following statements are equivalent:
	\begin{enumerate}[(i)]
		\item $\lambda\not\in\sigma(\A)$,
		\item $\Delta_\lambda$ has a bounded inverse,
	\end{enumerate}
	where $\Delta_{\lambda}:\mathcal{D}(A) \subset X_1 \to X_1 $ is given by $\Delta_\lambda\coloneqq\lambda-A-B(\lambda-D)^{-1}C$. 
	In this case the resolvent of $\A$ is given by
	\begin{align}\label{eq:inv}
  	(\lambda-\A)^{-1}=\begin{pmatrix} \Delta_\lambda^{-1} & \Delta_\lambda^{-1}B(\lambda-D)^{-1} \\
    (\lambda-D)^{-1}C\Delta_\lambda^{-1} & (\lambda-D)^{-1}\lbrack \Id_{X_2}+C\Delta_\lambda^{-1}B(\lambda-D)^{-1}\rbrack \end{pmatrix} \, .
  \end{align}
\end{lema}
\begin{proof}
	For $\lambda\not\in\sigma(D)$ we consider on $\mathcal{D}(A)\times\mathcal{D}(D)$ the decomposition
	\begin{align*}
		\lambda-\A&=\begin{pmatrix} \lambda-A & -B \\ -C & \lambda-D \end{pmatrix} = 
			\begin{pmatrix} \Id_{X_1} & -B(\lambda-D)^{-1} \\ 0 & \Id_{X_2} \end{pmatrix}	
			\begin{pmatrix} \Delta_{\lambda} 
		& 0 \\ -C & \lambda-D\end{pmatrix}	\\
		&=:R_\lambda\circ L_\lambda.
	\end{align*}
	The operator $R_\lambda$ has a bounded inverse, while, since $\lambda \not \in \sigma(D)$,  
	$L_\lambda$ has a bounded inverse if and only if the same holds for $\Delta_\lambda$.
	In this case the inverse of $\lambda-\mathbb A$ is
	\begin{align*}
		(\lambda-\A)^{-1}&=L_\lambda^{-1}\circ R_\lambda^{-1}\\
		&=\begin{pmatrix} \Delta_\lambda^{-1} & 0 \\ (\lambda-D)^{-1}C\Delta_\lambda^{-1}  & (\lambda-D)^{-1}\end{pmatrix}
			\begin{pmatrix} \Id_{X_1} & B(\lambda-D)^{-1} \\ 0 & \Id_{X_2} \end{pmatrix}\\
		&=\begin{pmatrix} \Delta_\lambda^{-1} & \Delta_\lambda^{-1}B(\lambda-D)^{-1} \\
       (\lambda-D)^{-1}C\Delta_\lambda^{-1} & (\lambda-D)^{-1}\lbrack C\Delta_\lambda^{-1}B(\lambda-D)^{-1}+\Id_{X_2}\rbrack \end{pmatrix}
	\end{align*}
	as we wanted to prove.
\end{proof}

It is clear from the proof that a similar statement is valid considering the Schur complement with respect to $A$ instead of $D$.

We are now ready to prove Theorem  \ref{theo:2by2}.

\begin{proof}[Proof of Theorem \ref{theo:2by2}]
	Let us consider $\lambda \in \C \setminus \big(\sigma(A)\cup \sigma(D)\big)$ such that $\mathcal{R}_{21}(\lambda)<1$. 
	Since $\lambda \not \in \sigma(D)$, Schur's Lemma~\ref{lem:schur} gives that $\lambda\not\in \sigma(\A)$ 
	if and only if the operator $\Delta_{\lambda}=\lambda-A-B(\lambda-D)^{-1}C$ has a bounded inverse. 
	This is actually the case. 
	Indeed, since $\lambda \not \in \sigma(A)$ we may write
		$$\Delta_{\lambda}= \Big(\Id_{X_1}- B (\lambda -D)^{-1} C (\lambda -A)^{-1}  \Big) (\lambda - A) \,$$
	on $\mathcal{D}(A)$. 
	The invertibility of $\Delta_{\lambda}$ and the fact that the inverse is bounded follows by the Neumann series criterion 
	and $\mathcal{R}_{21}(\lambda)<1$.
\end{proof}

%% file: setting_v7.tex
\section{Setting}

In the following we will always assume, without recalling it, that $n \in \N$ and $n\geq 2$.
Furthermore let $X_1,\ldots,X_n$ be complex Banach spaces and $X\coloneqq X_1\times\ldots\times X_n$. 
Of course, all $\ell^p$-norms on the product space $X$ are equivalent, but we will focus on the $1$-norm, 
i.e., we will always tacitly take
\begin{eqnarray*}
  &\|x\|\coloneqq\sum_{i=1}^n \|x_i\|_{X_i}\, , \qquad & \hbox{for } x=(x_1,\ldots,x_n)\in X\ .
\end{eqnarray*}

In the rest of the work we impose the following.
\begin{assum}\label{Ass1}
	For $i,j \in \{ 1\, .., n\}$, $A_{ij}:\mathcal{D}(A_{ij})\subset X_j\rightarrow X_i$ are linear operators 
	such that $A_{ii}$ are closed on $\mathcal{D}(A_{ii})$ and $A_{ij}$, for $i \ne j$, is relatively $A_{jj}$-bounded. 
	The associated operator matrix
		$$\A\coloneqq (A_{ij})_{i,j=1}^n:\mathcal{D}(\A) \subset X\rightarrow X \, ,$$
	is defined on $\mathcal{D}(\A):=  \mathcal{D}(A_{11}) \times .... \times  \mathcal{D}(A_{nn})$. 
\end{assum}

As usual, we denote by $\|\A\|$ the operator norm of $\A$, for $\A$ as in Assumption \ref{Ass1}. 
This norm depends on the norm that we have choosen on the product space $X$.
\begin{lema}\label{lem:sums}
	Let $\A$ be an operator matrix as in Assumption \ref{Ass1} acting on $(X,\|\cdot\|)$ and denote by $\|\A\|_{1}$ its operator norm. 
	Then,
  \begin{align*}
   		\|\A\|_{1}\leq\underset{j=1,\ldots,n}{\max}\sum_{i=1}^n\|A_{ij}\|.
  \end{align*}
\end{lema}
The proof is a direct computation. 
If all Banach spaces $X_i$ are one-dimensional, 
the inequality in Lemma \ref{lem:sums} is actually an equality, see e.g.~\cite[Aufgabe 114.4]{Heu04}.

Our spectral localisation result will exploit the following Gershgorin-type sets.

\begin{defi}\label{defi:S_i}
	Let $\A=(A_{ij})$ be an $(n\times n)$ operator matrix. 
	Consider for $1\le j,k\le n$, $j \ne k$, the \emph{Schur sets}
		\[
			S_{kj}(\A)\coloneqq\sigma(A_{kk})\cup\sigma(A_{jj})\cup
					\left\{\lambda\in\C\setminus\big(\sigma(A_{kk})\cup\sigma(A_{jj})\big)\ |\ {\mathcal R}_{kj}(\lambda) \geq 1\right\}\ ,
		\]
	where
	\begin{align}\label{eq:Rkj}
		{\mathcal R}_{kj}(\lambda) & :=\hspace{-.35cm}
		\sum_{\quad i=1, i\not=k}^{n} \Big\|\Big(A_{ik}(\lambda-A_{kk})^{-1}A_{kj}+ (1-\delta_{ij})A_{ij}\Big)(\lambda-A_{jj})^{-1} \Big\|\ ,
	\end{align}
	where $\delta_{ij}$ denotes the Kronecker-Delta, as well as
		\[
			S_k(\A)\coloneqq\hspace{-.35cm}\bigcup_{\quad j=1, j\not=k}^n	S_{kj} (\A)\ .
		\]
\end{defi}

Observe that in general $\mathcal R_{kj}(\lambda)\neq \mathcal R_{jk}(\lambda)$ even if $\A$ is self-adjoint.


\subsection{Relative boundedness of an operator vector}

In this work we are going to prove that via Schur's Lemma we can describe the spectrum of an arbitrarily large (finite) operator matrix.
Lemma~\ref{lem:schur} is formulated for $(2\times 2)$ operator matrices. 
Of course, one may apply it recursively by regarding the entries as operator matrices in their own right, 
but this would lead to less sharp estimates and, furthermore, 
it would force to impose unnatural relative boundedness 
conditions on such operator matrices, rather than on the elementary building blocks we are interested in. 
To begin with, we introduce a notation that better fits our framework.
 
\begin{defi}\label{defi:minor}
	Let $k\in \mathbb N$, $1\le k\le n$. 
	Let $\A=(A_{ij})$ be an $(n\times n)$ operator matrix satisfying Assumption \ref{Ass1}. 
	We denote by $\A_k$ its \emph{upper-left block} $\A_k\coloneqq(A_{ij})_{i,j=1}^k$, 
		\[
			\A_k:\mathcal{D}(\A_k) \subset X_1\times\ldots\times X_k \rightarrow X_1\times\ldots\times X_k \ ,
		\]
	with $\mathcal{D}(\A_k):=\mathcal{D}(A_{11})\times\ldots\times\mathcal{D}(A_{kk}) $.
\end{defi}

In the $(2\times 2)$ case it is sufficient to assume that $B$ (i.e. $A_{12}$) and $C$ (i.e. $A_{21}$) are relatively bounded with 
respect to $D$ and $A$ (i.e. $A_{22}$ and $A_{11}$) respectively. 
An $(n\times n)$ operator matrix will be treated as a $(2\times 2)$ matrix writing it as
\begin{align*}
	\A=\begin{pmatrix}
			\A_{n-1} & *\\ T & A_{nn}
			\end{pmatrix}\ .
	\end{align*}
We first need to understand under which assumptions $T$ is relatively $\A_{n-1}$-bounded. 
We do this in the next lemma.

\begin{lema}\label{lem:rel-bound}
	Let $\A$ be an $(n \times n)$ operator matrix satisfying Assumptions \ref{Ass1}. 
	Consider an operator vector
		\[
			T:=(T_1,\ldots,T_{n-1}):\mathcal{D}(T_{1})\times\ldots\times\mathcal{D}(T_{n-1})\rightarrow X_{n},
		\]
	 with $T_j$ relatively $A_{jj}$-bounded for all $j=1,\ldots,n-1$. 
	Furthermore let $\A_k$ be closed for $k=2,..,n$. 
	If $\sigma(A_{kk})\cup\sigma(\A_{k})\not=\C$ for $k\in\{2,\ldots,n-2\}$, then
	\begin{align*}
		T(\lambda-\A_{n-1})^{-1}\ is\ bounded  \qquad  \forall \lambda\not\in \sigma(A_{n-1,n-1})\cup\sigma(\A_{n-1}) \, .
	\end{align*}
	Hence, $T$ is relatively $\A_{n-1}$-bounded if in particular
	\begin{align}\label{restriction}
		\sigma(A_{kk})\cup\sigma(\A_k)\not=\C\qquad \forall k\in\{1,\ldots,n\}\setminus\{1,n\}\ .
	\end{align}
\end{lema}

\begin{proof}
	We proof the assertion by induction on $n\in \mathbb N$, $n\geq2$. \\
	For $n=2$ there is nothing to prove (and no condition as in \eqref{restriction} needs to be assumed), 
	since by assumption $T$ is relatively $A_{11}$-bounded and $A_{11}=\A_1$. 

	Let now $n\ge 3$ and let the statement be true up to $n-1$. 
	We regard $\A_{n-1}$ as a $(2\times 2)$ operator matrix writing
	\begin{align*}
		\A_{n-1}=\begin{pmatrix} A & B\\ C & D\end{pmatrix} \mbox{ with } &A \coloneqq \A_{n-2},\; 
				B \coloneqq \begin{pmatrix} A_{i,n-1}\end{pmatrix}_{i=1}^{n-2},\\
		& C \coloneqq (A_{n-1,j})_{j=1}^{n-2} \mbox{ and } D\coloneqq A_{n-1,n-1}\ .
	\end{align*}
	Define also
		\[
			\hat T := (T_1,\ldots,T_{n-2}):\mathcal{D}(T_{1})\times\ldots\times\mathcal{D}(T_{n-2})\rightarrow X_{n} \, .
		\]
	Since $\sigma(A_{kk})\cup\sigma(\A_{k})\not=\C$ for $k\in\{2,\ldots,n-2\}$,
	by induction hypothesis $C$ is relatively $A$-bounded and also $\hat{T}$ is relatively $A$-bounded. 
	Clearly $B$ is relatively $D$-bounded. 
	As in Lemma~\ref{lem:schur}, we consider, for $\lambda \not \in \sigma(D)$ (i.e. $\lambda \not \in \sigma(A_{n-1,n-1})$), the operator 
		\[
			\Delta_{\lambda}= \lambda-A-B(\lambda-D)^{-1}C: \mathcal{D}(A) \rightarrow  X_1\times \ldots \times X_{n-2}\ .
		\]
	For $\lambda\not\in\left(\sigma(A_{n-1,n-1})\cup\sigma(\A_{n-1})\right)$ by Lemma~\ref{lem:schur} 
	$\Delta_{\lambda}$ has a bounded inverse $\Delta_{\lambda}^{-1}$.  
	It follows from \eqref{eq:inv} and writing 
		\[
			T:=(\hat T , T_{n-1})\ ,
		\]
	that
	\begin{align*}
		& T(\lambda-\A_{n-1})^{-1} =(V_1, V_2) 
	\end{align*}
	with 
	\allowdisplaybreaks{
		\begin{align*}
			V_1 &: = \hat T \Delta_\lambda^{-1} + T_{n-1} (\lambda-D)^{-1} C \Delta_\lambda^{-1} , \\
			V_2 & := \hat T \Delta_\lambda^{-1} B (\lambda-D)^{-1}+T_{n-1} (\lambda-D)^{-1} 
				\lbrack \Id_{X_{n-1}}+C \Delta_\lambda^{-1} B (\lambda-D)^{-1} \rbrack\, .
		\end{align*}}
	Since $\hat{T}$ is relatively $A$-bounded, there exist positive constants $\alpha$ and $\beta$ 
	such that for $x \in X_1\times .. \times X_{n-2}$
	\begin{align*}
		\| \hat T \Delta_{\lambda}^{-1} x \|_{X_n} & \leq  \alpha \|\Delta_{\lambda}^{-1} x \|_{X_1\times .. \times X_{n-2}} + 
				\beta \| A \Delta_{\lambda}^{-1} x \|_{X_1\times .. \times X_{n-2}} \\
		& \leq \max\{\alpha,\beta\} \|\Delta_{\lambda}^{-1} x \|_{\mathcal{D}(A)} \leq \max\{\alpha,\beta\}
				\| \Delta_{\lambda}^{-1}  \| \|x \|_{X_1\times .. \times X_{n-2}} \, .
	\end{align*}
	Hence, $\hat T \Delta_\lambda^{-1}$ is bounded. 
	Likewise, $C\Delta_\lambda^{-1}$ is bounded. 
	As $T_{n-1}$ and $B$ are relatively $D$-bounded, we see that $T_{n-1} (\lambda-D)^{-1}$ and $B (\lambda-D)^{-1}$ are bounded.
	We conclude that  $T(\lambda-\A_{n-1})^{-1}$ is bounded. 
\end{proof}

%% file: main_results_v7.tex
\section{Main results}\label{sec:main}


\subsection{Spectral localisation by means of the Schur sets}

We are now in position to give an estimate for the spectrum of $\A$ using only the methods we derived from Schur's Lemma \ref{lem:schur}. 

\begin{theo}\label{theo:nagel-implies-gershgorin}
	Let $\A$ be an $(n \times n)$ operator matrix satisfying Assumptions \ref{Ass1} and further assume that 
	$\A_k$ are closed for all $k \in \{2, ...,n\}$. 
	If \eqref{restriction} holds, then
	\begin{align*}
		\sigma(\A)\subset S_n(\A) \, .
	\end{align*}
\end{theo}
\begin{proof}
	The case $n=2$ was already considered in Theorem \ref{theo:2by2}. 
	The proof in the general case is rather similar. 
	For this we split $\A$ into the blocks $A \coloneqq \A_{n-1}$, $B \coloneqq ( A_{in})_{i=1}^{n-1}$,
	$C \coloneqq (A_{nj})_{j=1}^{n-1}$, $D\coloneqq A_{nn}$ and write 
	\begin{align*}
		\A=\begin{pmatrix} A & B\\ C & D\end{pmatrix} \, .
	\end{align*} 
	By assumption $A$ with domain $\mathcal{D}(A_{11})\times\ldots\times\mathcal{D}(A_{n-1,n-1})$ is closed. 
	This together with Lemma \ref{lem:rel-bound} yield that $C$ is relatively $A$-bounded. 
	 
	Consider a $\lambda\not\in S_n(\A)$. Then $\lambda\not\in\sigma(D)=\sigma(A_{nn})$ and we can apply Schur's Lemma.
	We will show that  $\Delta_{\lambda}=\lambda - A-B(\lambda-D)^{-1}C$ has a bounded inverse, 
	which implies $\lambda\not\in\sigma(\A)$ by Lemma \ref{lem:schur}. 
	We denote by $L$ the operator matrix $L\coloneqq \diag(\lambda - A_{ii})_{i=1}^{n-1}$. 
	Since $\lambda \not \in S_n(\A)$, $L$ is invertible and we find
	\begin{align*}
		\Delta_{\lambda} & = \lambda - A-B(\lambda-D)^{-1}C\\
		& = L - \Big((1-\delta_{ij}) A_{ij}\Big)_{i,j=1}^{n-1} + \Big(A_{in}(\lambda-A_{nn})^{-1} A_{nj}\Big)_{i,j=1}^{n-1}\\
		&=[I- \Big( (1-\delta_{ij}) A_{ij}  +A_{in}( \lambda-A_{nn})^{-1} A_{nj}\Big)_{i,j=1}^{n-1} L^{-1}]L \\
		&=:[I-R]L,
	\end{align*}
	with an $\big((n-1)\times (n-1)\big)$ operator matrix $R:=(R_{ij})_{i,j=1}^{n-1}$ given by
	\begin{align*}
		R_{ij}&=\Big( A_{in}(\lambda-A_{nn})^{-1}A_{nj}+ (1-\delta_{ij})A_{ij} \Big) (\lambda-A_{jj})^{-1}. 
	\end{align*}
	Since $L$ is invertible, it remains to show that $I-R$ is invertible and for this we will once again use the Neumann series criterion. 	
	By Lemma \ref{lem:sums} and the definition of ${\mathcal R}_{kj}(\lambda)$ (see \eqref{eq:Rkj}), the norm of $R$ can be estimated by
	\begin{align*}
		\|R\|&\leq\max_{j=1,\ldots,n-1}\left\{{\mathcal R}_{nj}(\lambda) \right\} < 1 \, ,
	\end{align*}
	since $\lambda \not \in S_{n}(\A)$ and hence $\lambda \not \in S_{nj}(\A)$ for all $j \in \{1, \ldots, n-1\}$. 
	Thus $\Delta_{\lambda}$ has a bounded inverse as a composition of operators with bounded inverses. 
	The claim follows.
\end{proof}

\begin{rema}
	If we consider $(X, \| \cdot\|_{\infty})$, then the analogous to Lemma \ref{lem:sums} is
	\begin{align*}
				\|\A\|\leq\underset{i=1,\ldots,n}{\max}\sum_{j=1}^n\|A_{ij}\| \, ,
		 \end{align*}
	and hence the condition for the invertibility of $I -R$ would be 
	$\underset{i=1,\ldots,n-1}{\max}\big\{\tilde{{\mathcal R}}_{ni}(\lambda) \big\} < 1$ where
	\begin{align*}
		\tilde{{\mathcal R}}_{ni}(\lambda) & =
		\sum_{\substack{j=1}}^{n-1} \Big\|\Big(A_{in}(\lambda-A_{nn})^{-1}A_{nj}+ (1-\delta_{ij})A_{ij}\Big)(\lambda-A_{jj})^{-1} \Big\| \, .
	\end{align*}
\end{rema}

Since the spectrum is invariant under permutations, one may improve Theorem~\ref{theo:2by2} by also observing the inclusion 
	\[
		\sigma(\mathbb A)\subset S_{12}(\mathbb{A})
	\] 
and therefore by the distributive property of union over intersection
	\[
		\sigma(\mathbb A)\subset \sigma(A)\cup \sigma(D) \cup \{\lambda\in \mathbb C\setminus  \sigma(A)\cup \sigma(D) : 
		\mathcal R_{21}(\lambda)\ge 1\hbox{ and }\mathcal R_{12}(\lambda)\ge 1\}\ ,
	\] 
where
	\[
		\mathcal R_{12}(\lambda):=\|C(\lambda-A)^{-1}B(\lambda-D)^{-1}\|\quad\hbox{and}\quad \mathcal R_{21}(\lambda):=
		\|B(\lambda-D)^{-1}C(\lambda-A)^{-1}\|\ .
	\]
Likewise, it is possible to improve the estimate in Theorem~\ref{theo:nagel-implies-gershgorin} by permuting the order of the $X_i$'s,
as the $n$-th column and row play a special role in Theorem \ref{theo:nagel-implies-gershgorin}. 
For this the following definition is useful.

\begin{defi}\label{defi:Aperm} 
	Let $\pi$ be a permutation of $\{ 1, ..,n\}$. 
	Given an $(n \times n)$ operator matrix $\A=(A_{ij})$ satisfying Assumptions \ref{Ass1}, 
	the \emph{permuted matrix} $\A(\pi)$ is the $(n \times n)$ operator matrix obtained by first permuting the rows and 
	then permuting the columns of $\A$, according to $\pi$. 
	That is, 
		$$\A(\pi):\mathcal{D}(\A(\pi)) \subset X_{\pi^{-1}(1)}\times\ldots\times X_{\pi^{-1}(n)} 
		\rightarrow   X_{\pi^{-1}(1)}\times\ldots\times X_{\pi^{-1}(n)}\, ,$$
	with $\mathcal{D}(\A(\pi)):=  \mathcal{D}(A_{\pi^{-1}(1),\pi^{-1}(1)}) \times .. \times  \mathcal{D}(A_{\pi^{-1}(n),\pi^{-1}(n)})$ and
		$$\A(\pi):=(A_{\pi^{-1}(i), \pi^{-1}(j)})_{i,j=1}^n \, .$$  
\end{defi}

\begin{rema}
	Since $\A$ satisfies Assumptions \ref{Ass1}, so does the permuted matrix	 $\A(\pi)$. 
	Moreover, if $\A$ is closed then also $\A(\pi)$ is closed and $\sigma(\A)=\sigma(\A(\pi))$. 
	This can be easily seen writing $\A(\pi)=P \A P^{t}$ where $P$ is the $(n \times n)$ permutation matrix 
		$$P=(P_{ij})_{i,j=1}^n:  X_{1}\times\ldots\times X_{n} \rightarrow X_{\pi^{-1}(1)}\times\ldots\times X_{\pi^{-1}(n)},$$
	with $P_{ij}=0$ for $i\not=\pi(j)$ and $P_{\pi(j),j}=\Id_{X_j}$. 
	Similarly, if $\A$ is self-adjoint then so is $\A(\pi)$.
\end{rema}

We then obtain the following refinement of Theorem~\ref{theo:nagel-implies-gershgorin}.

\begin{coro}\label{cor:nagel-implies-gershgorin}
	Let $\A$ be an $(n \times n)$ operator matrix satisfying Assumptions \ref{Ass1}. 
Then
		\begin{align*}
			\sigma(\A)\subset \bigcap_{m\in  \mathfrak I} S_m(\A)\ ,
		\end{align*}
	where $\mathfrak I$ is the set of those $m\in \{1,\ldots,n\}$ 
	such that there is a permutation $\pi$ of the set $\{1,\ldots,n\}$ for which 
	\begin{enumerate}
		\item $\pi(m)=n$;
		\item $\A(\pi)_{k}$ is closed for all $k\in\{1,\ldots,n\}\setminus\{1,n\}$;
		\item $\sigma\big(A_{\pi^{-1}(k),\pi^{-1}(k)}\big)\cup\sigma\big(\A(\pi)_k\big)\not=\C$ for all $k\in\{1,\ldots,n\}\setminus \{1,n\}$.
	\end{enumerate}
\end{coro}
\begin{proof}
	Let $m$ be an element of $ \mathfrak I$ and $\pi$ be a permutation of $\{1, .., n\}$ as in the definition of $\mathfrak I$. 
	Then $\A(\pi)$ fullfills the hypothesis of Theorem \ref{theo:nagel-implies-gershgorin} and hence
	\begin{align*}
		\sigma(\A)=\sigma(\A(\pi))&\subset S_n(\A(\pi))	\ . 
	\end{align*}
	By the definition of $\mathcal{R}_{nj}(\lambda)$ in \eqref{eq:Rkj} 
	(adding a dependence on the matrix in the notation), one sees that
	\begin{equation}\label{eq:auf} 
		\mathcal{R}_{nk} (\lambda;\A(\pi)) =  \mathcal{R}_{mj} (\lambda;\A) \mbox{ if }m\ne j \mbox{ and }\pi(m)=n, \pi(j)=k \, .
	\end{equation}
	Hence, $\sigma(\A)\subset S_n(\A(\pi))	  = S_m(\A)$. 
	Consequently the spectrum of $\A$ is contained in the intersection of these sets, as we wanted to prove.
\end{proof}

\begin{rema}
	1) Condition (3) in the definition of $\mathfrak I$ is surely satisfied if $\A$ is self-adjoint. 
	If on the other hand all entries of $\mathbb A$ are bounded, then by definition $\mathfrak I=\{1,\ldots,n\}$.	\\
	2) Observe that Corollary~\ref{cor:nagel-implies-gershgorin} is actually compatible with Theorem~\ref{theo:nagel-implies-gershgorin},
	in the sense that if the assumptions of the latter hold, then at least the identity is an allowed permutation and 
	thus $n$ is an element of $\mathfrak I$.

3) From Equation \eqref{eq:auf}, it is clear that given two permutations $\pi_1$ and $\pi_2$ of $\{1,\ldots,n\}$ with 
$\pi^{-1}_1(n)=\pi^{-1}_2(n)$, then $S_n(\A(\pi_1))=S_n(\A(\pi_2))$. 
But it could be that Condition \textit{(3)} in Corollary 
\ref{cor:nagel-implies-gershgorin} is satisfied for only one of these two permutations. 
This is the reason for allowing arbitrary permutations, instead of restricting to permutations that 
interchange only one of the rows/columns with the $n$-th row/column.
\end{rema}


\subsection{Spectral localisation by means of the Cassini ovals}

Cassini ovals for operator matrices of bounded linear operators are studied in~\cite[\S~5]{HerSch07}. 
It is easy to check that the proof in~\cite[Thm.~5.1]{HerSch07} remains valid if the diagonal entries are merely closed. 
However, it is not clear whether it can be adapted to the case of unbounded off-diagonal entries. 
Furthermore, Cassini-type inclusions have been proved in~\cite{HerSch07} merely 
for the \emph{approximate point spectrum} of such operator matrices: 
We are going to sharpen said spectral localisation as a consequence of the results in the previous section.

\begin{defi}\label{defi:Cassini}
	Let $\A=(A_{ij})$ be an $(n\times n)$ operator matrix satisfying Assumptions \ref{Ass1}. 
	For $1\le i,j\le n$, $i \ne j$, the \emph{Cassini ovals} are the sets
		\[
			C_{ij}(\A)\coloneqq\sigma(A_{ii})\cup\sigma(A_{jj})\cup \widetilde{C}_{ij}(\A)\ ,
		\]
	where
	\begin{align*}
		\widetilde{C}_{ij}(\A)&\coloneqq\Big\{\lambda\in\C\ \Big\vert\ \lambda\not\in\sigma(A_{ii})\cup\sigma(A_{jj}) \mbox{ and }\\ 
				& \qquad \qquad \quad \Big(\hspace{-.35cm}\sum_{\quad l=1, l\not=i}^n\|A_{li}(\lambda-A_{ii})^{-1}\|\Big)
				\Big(\hspace{-.35cm}\sum_{\quad l=1, l\not=j}^n\|A_{lj}(\lambda-A_{jj})^{-1}\|\Big)\geq1\Big\}.
	\end{align*}
\end{defi}
It is clear from the definition that $C_{ij}(\A)=C_{ji}(\A)$ and that, by sub-multiplicativity of the norm, 
the Cassini ovals are contained in the Gershgorin disks defined in \eqref{eq:Gi}.

\begin{theo}\label{theo:nagel-implies-cassini}
	Let $\A$ be an $(n \times n)$ operator matrix satisfying Assumptions \ref{Ass1}. 
	Assume that for any permutation $\pi$ of $\{ 1, \dots,n\}$, $\A(\pi)_k$ is closed for any $k\in\{1,\ldots,n\}\setminus\{1,n\}$. 
	Then,
	\begin{align}\label{eq:CA}
		\sigma(\A)\subset C(\A)\coloneqq\bigcup_{1 \leq i <j\leq n} C_{ij}(\A)\ . 
	\end{align}
\end{theo}
\begin{proof}
  The proof is by induction. 
 	
	The case $n=2$ follows from Theorem \ref{theo:2by2}. Indeed, by that result and the sub-multiplicativity of the norm
	\begin{align*}
		\sigma(\A) &  \subset  \sigma(A_{11}) \cup \sigma(A_{22}) \\
		& \qquad \cup \{ \lambda \in \C\setminus \big(\sigma(A_{11})\cup \sigma(A_{22})\big): 
			\| A_{12} (\lambda -A_{22})^{-1} A_{21} (\lambda -A_{11})^{-1}\| \geq 1\} \\
		& \subset  \sigma(A_{11}) \cup \sigma(A_{22}) \cup \widetilde{C}_{12}(\A) = C_{12}(\A) \, .
	\end{align*}

	Let now the statement be true for a $n-1\in\N$, $n>2$. 
	Define $\mathfrak I$ as in Corollary \ref{cor:nagel-implies-gershgorin}. 
	It is convenient to separate the cases $\mathfrak I\not=\{1,\ldots,n\}$ and   $\mathfrak I=\{1,\ldots,n\}$.\\
	Assume first that $\mathfrak I\not=\{1,\ldots,n\}$. 
	Then, by assumption, there exists a permutation $\pi$ of $\{1 , ..,n\}$ such that 
	$\sigma(A_{\pi(k),\pi(k)})\cup\sigma(\A(\pi)_k)=\C$ for some $k \in \{1,\ldots, n\}\setminus \{1,n\}$. 
	By induction hypothesis we get 
	\begin{align*}
		\sigma(\A(\pi)_k)&\subset\bigcup_{1 \leq i <j\leq k} C_{ij}(\A(\pi)_k)
				\subset\bigcup_{1 \leq i <j\leq k} C_{\pi^{-1}(i),\pi^{-1}(j)}(\A)	
				\subset\bigcup_{1 \leq i <j\leq n} C_{ij}(\A)\ .
	\end{align*}
	Whereas the inclusion $C_{ij}(\A(\pi)_k) \subset C_{\pi^{-1}(i),\pi^{-1}(j)}(\A)$ holds observing that
	\allowdisplaybreaks{
		\begin{align*}
			&  \hspace{-.35cm}\sum_{\quad l=1,l\not=i}^k\|A_{\pi^{-1}(l),\pi^{-1}(i)}(\lambda-A_{\pi^{-1}(i),\pi^{-1}(i)})^{-1}\| \\
			& \quad \leq
				\hspace{-.35cm} \sum_{\quad l=1,l\not=\pi^{-1}(i)}^n\|A_{l,\pi^{-1}(i)}(\lambda-A_{\pi^{-1}(i),\pi^{-1}(i)})^{-1}\| \, .
		\end{align*}
	}
	By definition of the Cassini ovals, it is clear that $\sigma(A_{\pi(i),\pi(i)})\subset C(\A)$ for any $i \in \{1,\ldots, n\}$. 
	Thus we get 
	\begin{align*}
		\C=\sigma(A_{\pi(k),\pi(k)})\cup\sigma(\A(\pi)_k)\subset \bigcup_{1 \leq i <j\leq n} C_{ij}(\A) = C(\A) ,
	\end{align*}
	and the claim follows trivially.\\
	If $\mathfrak I=\{1,\ldots,n\}$, we consider a $\lambda\in\bigcap_{k=1}^n S_{k}(\A)$ with $\lambda\not\in C (\A)$. 
	We will show that such a $\lambda$ cannot exist. 
	Notice that $\lambda\not\in C (\A)$ implies that $\lambda \not \in \sigma(A_{ii})$ for all $i \in \{1, ..,n\}$. 
	We will use this in the following without further noticing it. 
	Since $\lambda\not\in C (\A)$, by definition $\lambda\not\in \widetilde{C}_{ij}(\A)$ for all $1 \leq i<j\leq n$, 
	which implies for all such indices
		\[
			\Big(\hspace{-.35cm} \sum_{\quad l=1,l\not=i}^n\|A_{li}(\lambda-A_{ii})^{-1}\|\Big)
			\Big(\hspace{-.35cm} \sum_{\quad l=1,l\not=j}^n\|A_{lj}(\lambda-A_{jj})^{-1}\|\Big)<1 . 
		\]
	Hence there exists $m\in\{1,\ldots,n\}$ such that
		\[
			\hspace{-.35cm}\sum_{\quad l=1,l\not=j}^n\|A_{lj}(\lambda-A_{jj})^{-1}\|<1 \qquad \forall j\not=m\ .
		\]
	Consider now $k \ne m$. 
	Then using the notation in~\eqref{eq:Rkj} we see that for $j \ne m,k$
	\begin{align*}
		\mathcal{R}_{kj}(\lambda) & \leq \hspace{-.35cm} \sum_{\quad i=1,i\not=k}^{n}\Big( 
			\|A_{ik}(\lambda-A_{kk})^{-1}\| \|A_{kj}(\lambda-A_{jj})^{-1} \| \Big)+\hspace{-.5cm} 
			\sum_{\quad i=1,i\not=k,j}^{n} \|A_{ij}(\lambda-A_{jj})^{-1}\|\\
		& < \|A_{kj}(\lambda-A_{jj})^{-1} \|+ \hspace{-.5cm} \sum_{\quad i=1,i\not=k,j}^{n} \|A_{ij}(\lambda-A_{jj})^{-1}\| <1 \, .
	\end{align*}
	Hence, $\lambda \not \in S_{kj}(\A)$ for all $j \ne m,k$. 
	On the other hand, $\lambda \in S_k(\A)$ and hence necessarily $\lambda \in S_{km}(\A)$. 
	It then holds
	\begin{align*}
		1&\leq \mathcal{R}_{km}(\lambda) \\
		& \leq \Big(\hspace{-.35cm}\sum_{\quad i=1,i\not=k}^{n}\|A_{ik}(\lambda-A_{kk})^{-1}\|\Big)\|A_{km}(\lambda-A_{mm})^{-1}\|
				+\hspace{-.5cm} \sum_{\quad i=1,i\not=m,k}^{n}\|A_{i m}(\lambda-A_{mm})^{-1}\|	\\
		&<\hspace{-.35cm} \sum_{\quad i=1,i\not=m}^{n}\|A_{i m}(\lambda-A_{mm})^{-1}\|.
	\end{align*}
	Since $\lambda\in S_m(\A)$, there is a $j\not=m$ with $\lambda\in S_{mj}(\A)$ which implies
	\begin{align*}
		1&\leq \mathcal{R}_{mj}(\lambda) \\
		& \leq \hspace{-.35cm} \sum_{\quad i=1,i\not=m}^{n}\|A_{im}(\lambda-A_{mm})^{-1}\| \|A_{mj}(\lambda-A_{jj})^{-1}\|
				+ \hspace{-.5cm} \sum_{\quad i=1,i\not=m,j}^{n}\|A_{ij}(\lambda-A_{jj})^{-1}\|	\\
		& < \Big(\hspace{-.35cm} \sum_{\quad i=1,i\not=m}^{n}\|A_{im}(\lambda-A_{mm})^{-1}\| \Big) 
			\Big( \|A_{mj}(\lambda-A_{jj})^{-1}\|+ 	\hspace{-.5cm} \sum_{\quad i=1,i\not=m,j}^{n}\|A_{ij}(\lambda-A_{jj})^{-1}\| \Big) \\
		& = \Big( \hspace{-.35cm} \sum_{\quad i=1,i\not=m}^{n}\|A_{im}(\lambda-A_{mm})^{-1}\| \Big)  \Big( \hspace{-.35cm} 
			\sum_{\quad i=1,i\not=j}^{n}\|A_{ij}(\lambda-A_{jj})^{-1}\| \Big) <1 \, ,
	\end{align*}
	since $\lambda\not\in \widetilde{C}_{mj}(\A)$, a contradiction. 
	Hence, $\bigcap_{k=1}^n S_{k}(\A)\subset C (\A)$ and the claim follows in this case from Corollary \ref{cor:nagel-implies-gershgorin}.
\end{proof}

We stress that Theorem~\ref{theo:nagel-implies-cassini} is valid without assuming condition~\eqref{restriction} to be satisfied.

%% file: other_results_v7.tex
\section{Other Results}


\subsection{The modified Schur sets}

The Schur sets defined in Definition \ref{defi:S_i} are the natural ones to consider in order to describe the spectrum 
by means of Schur's Lemma. 
However, we are now going to present an alternative localization result based on a new family of sets: 
These are in general bigger then the Schur sets but, as we will see, 
allow for estimates that do not depend on condition~\eqref{restriction}. 
Of course, this is useful only when some entries of the operator matrix might in fact have unbounded spectra.

\begin{defi}\label{defi:S_istar}
	Let $\A=(A_{ij})$ be an $(n\times n)$ operator matrix. 
	Consider for $1\le j,k\le n$, $j\ne k$, the \emph{modified Schur sets}
		\[
			S^*_{kj}(\A)\coloneqq\sigma(A_{kk})\cup\sigma(A_{jj})\cup
				\left\{\lambda\in\C\setminus\big(\sigma(A_{kk})\cup\sigma(A_{jj})\big)\ |\ {\mathcal R}^*_{kj}(\lambda) \geq 1\right\}\ ,
		\]
	where
	\begin{align*}
		{\mathcal R}^*_{kj}(\lambda) & :=
			\hspace{-.35cm} \sum_{\quad i=1,i\not=k}^{n}  \hspace{-.2cm}\Big( \big\|A_{ik}(\lambda-A_{kk})^{-1}A_{kj}(\lambda-A_{jj})^{-1}\big\|+				\big\| (1-\delta_{ij})A_{ij}(\lambda-A_{jj})^{-1} \big\| \Big) ,
	\end{align*}
	as well as
		\[
			S_k^*(\A)\coloneqq \hspace{-.35cm}\bigcup_{\quad j=1,j\not=k}^n	 \hspace{-.15cm} S^*_{kj} (\A)\ .
		\]
\end{defi}
It is clear from the definition that $S_k(\A)\subset S_k^*(\A)$.

\begin{theo}\label{theo:no-restriction}
	Let $\A$ be an $(n \times n)$ operator matrix satisfying Assumptions \ref{Ass1}. 
	Assume that for any permutation $\pi$ of $\{ 1, \dots,n\}$, $\A(\pi)_k$ is closed for any $k \in \{2, ..., n\}$. 
	Then,
	 \begin{align*}
			\sigma(\A)&\subset\bigcap_{k=1}^n S^*_k(\A) \, .
	 \end{align*}
\end{theo}
\begin{proof}
	We first show that $\sigma(\A) \subset S^*_n(\A)$. 
	Since $S_n(\A)\subset S^*_n(\A)$, the inclusion follows from Theorem \ref{theo:nagel-implies-gershgorin} once we have shown that
	\begin{align}\label{eq:rest}
		\big(\sigma(A_{kk})\cup\sigma(\A_k) \big)\subset S^*_n(\A) \mbox{ for all }k \in \{2, ..,n-1\}\,.
	\end{align}
	In fact, from this one sees that if \eqref{restriction} is false, then $S_n^*(\A)=\C$ and $\sigma(\A)\subset S^*_n(\A)$ is then trivial.\\
	Instead of proving \eqref{eq:rest} directly, we first observe that by Theorem \ref{theo:nagel-implies-cassini} 
	and the definition of the Cassini ovals one has for $2 \leq k \leq n-1$
	\begin{align*}
		\sigma(\A_k)\subset\bigcup_{1 \leq i <j\leq k}C_{ij}(\A_{k})&\subset\bigcup_{1 \leq i <j\leq n-1}C_{ij}(\A_{n-1}) = C(\A_{n-1})\ .
	\end{align*}
	and hence
	\begin{align*}
		\big(\sigma(A_{kk})\cup\sigma(\A_k) \big)&\subset C(\A_{n-1})\mbox{ for all }k \in \{2, \ldots,n-1\}\, .
	\end{align*}  
	As by definition $\sigma(A_{kk})\subset S^*_n(\A)$ for $k=\{1,\ldots,n\}$, 
	we consider now $\lambda\in C(\A_{n-1})$ with $\lambda\not\in\sigma(A_{kk})$ for all $k=\{1,\ldots,n\}$ 
	and prove that $\lambda \in S_n^*(\A)$. 
	Since $\lambda \in  C(\A_{n-1})$, there exist indices $1 \leq i < j \leq n-1$ such that $\lambda\in \widetilde{C}_{ij}(\A)$, i.e.,
	\begin{align*}
		\Big( \hspace{-.35cm} \sum_{\quad l=1,l\not=i}^{n-1}\|A_{li}(\lambda-A_{ii})^{-1}\|\Big)
			\Big(\hspace{-.35cm} \sum_{\quad l=1,l\not=j}^{n-1}\|A_{lj}(\lambda-A_{jj})^{-1}\|\Big)\geq 1.
	\end{align*}
	Without loss in generality we may say that
	\begin{align}\label{eq:concon}
		\sum_{\quad l=1,l\not=i}^{n-1}\|A_{li}(\lambda-A_{ii})^{-1}\| \geq 1.
	\end{align}
	Since
	\begin{align*}
		\mathcal{R}^*_{ni}(\lambda) & \geq \sum_{l=1}^{n-1} \big\| (1-\delta_{li})A_{li}(\lambda-A_{ii})^{-1} \big\| \, ,
	\end{align*}
	the inequality in \eqref{eq:concon} implies $\mathcal{R}^*_{ni}(\lambda)\geq 1$ and consequently $\lambda\in S^*_n(\A)$.\\ 
	With the same arguments, considering permutations $\pi$ and the matrices $\A(\pi)$, 
	one finds $\sigma(\A) \subset S_k^*(\A)$ for all $k \in \{1, .., n\}$. 
	The claim follows.
\end{proof}


\subsection{A convenient set of assumptions}

In the main results we have to assume that $\A(\pi)_k$ is closed for all $k \in \{ 2, ..,n\}$ and any permutation $\pi$ of $\{ 1, ..,n\}$. We give now a set of assumptions on the off-diagonal operators $A_{ij}$ that assures the closedness of the upper-left blocks $\A(\pi)_k$.

\begin{assum}\label{Ass2}
	Let $\A$ be an $(n \times n)$ operator matrix satisfying Assumption \ref{Ass1}. 
	For $i,j \in \{1, ..,n\}$, $i \ne j$, there exists non-negative constants $c_{ij}$, $d_{ij}$ such that  
	\begin{equation}\label{eq:bs1}
		\| A_{ij} x_{j}\|_{X_i} \leq c_{ij} \| A_{jj} x_{j}\|_{X_j} + d_{ij} \| x_{j}\|_{X_j}\, ,
	\end{equation}
	for all $x_j \in \mathcal{D}(A_{jj})$. 
	Moreover, 
	\begin{equation}\label{eq:bs2}
		\sum_{\quad i=1,i \ne j}^n \hspace{-.1cm}c_{ij} <1 \, . 
	\end{equation}
\end{assum}

\begin{lema}\label{lem:closed}
	Let $\A$ be an $(n \times n)$ operator matrix satisfying Assumptions \ref{Ass1} and \ref{Ass2} 
	and $\pi$ be a permutation of $\{1, ..,n\}$. 
	Then $\A(\pi)_k$ is closed on $\mathcal{D}(\A(\pi)_k)$ for all $k \in \{1, ..,n\}$. 
\end{lema}
\begin{proof}
	Let $k \in \{1,\ldots,n\}\setminus \{1,n\}$ be fixed. 
	We consider first the case that $\pi$ is the identity. 
	Let $D:= (\delta_{ij} A_{ij})_{i,j=1}^k$ be the $(k \times k)$ operator matrix 
		$$D: \mathcal{D}(D) \subset X_1 \times .. \times X_{k} \to X_1 \times .. \times X_{k}\, ,$$ 
	with $\mathcal{D}(D)= \mathcal{D}(\A_{k})$. 
	We first observe that $D$ is closed. Indeed, for $x \in\mathcal{D}(D)$ we find
	\begin{align*}
		\| x \|_{\mathcal{D}(D)} = \sum_{i=1}^k \| x_i \|_{X_i} + \sum_{i=1}^k \| A_{ii}x_i \|_{X_i}  = \sum_{i=1}^k \| x_i \|_{D(A_{ii})} \, .
	\end{align*}
	Since $A_{ii}$ are closed, each $\mathcal{D}(A_{ii})$ is a Banach space and hence so is $\mathcal{D}(D)$. 
	It follows that $D$ is closed.\\
	We prove that $\A_k$ is a closed operator by proving that the induced graph-norm is equivalent to $\| \cdot \|_{\mathcal{D}(D)}$. 
	Notice that $\mathcal{D}(D)= \mathcal{D}(\A_{k})$. 
	For $x \in \mathcal{D}(\A_k)$ we find by Assumptions \ref{Ass2}
	\begin{align*}
		\| x\|_{\mathcal{D}(\A_k)} & \leq  \sum_{i=1}^k \| x_i \|_{X_i} + \sum_{i,j=1}^k \| A_{ij}x_j \|_{X_i} \\
		& \leq \sum_{i=1}^k \| x_i \|_{X_i} + \sum_{j=1}^k \Big( \| A_{jj}x_j \|_{X_j} + 
			\sum_{\substack{i=1\\i \ne j}}^k \| A_{ij}x_j \|_{X_i} \Big) \\
		& \leq \sum_{i=1}^k \Big(  1+ \hspace{-.35cm}  \sum_{\quad j=1,j \ne i}^k \hspace{-.15cm} d_{ji} \Big)   \| x_i \|_{X_i} + 
			\sum_{j=1}^k \Big(  1+\hspace{-.35cm}  \sum_{\quad j=1,j \ne i}^k \hspace{-.15cm} c_{ij} \Big) \| A_{jj}x_j \|_{X_j}  \\
		& \leq C_1 \| x \|_{\mathcal{D}(D)} \, ,
	\end{align*}
	with a strict positive constant $C_1$ depending only on the coefficients $c_{ij}, d_{ij}$ in Assumptions \ref{Ass2}. 
	Similarly, we have the following estimate from below. 
	Define
		$$ \delta:= 1 + \max_{i=1, ..,k} \hspace{-.35cm}  \sum_{\quad j=1,j \ne i}^k \hspace{-.1cm}  d_{ji} \, .$$
	Then, 
	\allowdisplaybreaks{
		\begin{align*} 
			& \delta  \sum_{i=1}^k \| x_i \|_{X_i} + \sum_{i=1}^k \| \sum_{j=1}^k A_{ij}x_j \|_{X_i} \\
			& \geq \delta \sum_{i=1}^k \| x_i \|_{X_i} + \sum_{i=1}^k \Big( \| A_{ii}x_i \|_{X_i} - \hspace{-.35cm}  
				\sum_{\quad j=1,j \ne i}^k \hspace{-.15cm} \| A_{ij}x_j \|_{X_i} \Big) \\
			& \geq \sum_{i=1}^k \Big(  \delta - \hspace{-.35cm}  \sum_{\quad j=1,j \ne i}^k \hspace{-.15cm} d_{ji} \Big)  \| x_i \|_{X_i} + 
				\sum_{j=1}^k \Big(  1 - \hspace{-.35cm}  \sum_{\quad j=1,j \ne i}^k \hspace{-.15cm} c_{ij} \Big) \| A_{jj}x_j \|_{X_j}   
				\geq C_2 \| x \|_{\mathcal{D}(D)} \, ,
		\end{align*}
	}
	with 
		$$C_2:= \min_{j=1, ..,k} \Big(1-\hspace{-.35cm}  \sum_{\quad j=1,j \ne i}^k \hspace{-.15cm} c_{ij}\Big) >0 \, ,$$
	by assumption. 
	Since
		$$\| x\|_{\mathcal{D}(\A_k)} \geq \frac{1}{\delta} \Big( \delta  \sum_{i=1}^k \| x_i \|_{X_i} + 
			\sum_{i=1}^k \| \sum_{j=1}^k A_{ij}x_j \|_{X_i} \Big)\, , $$
	we see that  $\| \cdot \|_{\mathcal{D}(\A_k)}$ and  $\| \cdot \|_{\mathcal{D}(D)}$ are equivalent 
	and hence $\A_k$ is closed on its domain.

	When $\pi$ is a general permutation of $\{1 , .., n\}$, the claim follows with the same arguments since in this case 
	the elements in the diagonal of $\A(\pi)_k$ are also closed operators and \eqref{eq:bs1}, \eqref{eq:bs2} still hold for $\A(\pi)$, too.   
\end{proof}

Accordingly, we can state a weaker but simple version of our main result.

\begin{theo}
	Let $\A$ be an $(n \times n)$ operator matrix satisfying Assumptions \ref{Ass1} and \ref{Ass2}.
	Let $\mathfrak I_0$ be the set of those $m\in \{1,\ldots,n\}$ such that 
	there is a permutation $\pi$ of the set $\{1,\ldots,n\}$ for which 
	\begin{enumerate}
		\item $\pi(m)=n$;
		\item $\sigma\big(A_{\pi^{-1}(k),\pi^{-1}(k)}\big)\cup\sigma\big(\A(\pi)_k\big)\not=\C$ for all $k\in\{2,\ldots,n-1\}$.
	\end{enumerate}
	Then
	\begin{align*}
		\sigma(\A)\subset \bigcap_{m\in \mathfrak I_0} S_m(\A)\ .
	\end{align*}
\end{theo}
\begin{proof}
	The claim follows from Corollary \ref{cor:nagel-implies-gershgorin} and Lemma \ref{lem:closed}.
\end{proof}

\begin{theo}
	Let $\A$ be an $(n \times n)$ operator matrix satisfying Assumptions \ref{Ass1} and \ref{Ass2}. Then 
	\begin{align*}
		\sigma(\A) \subset C(\A) \mbox{ and }\sigma(\A) \subset \bigcap_{k=1}^n S^*_k(\A) \, .
	\end{align*}
\end{theo}
\begin{proof}
	By Lemma \ref{lem:closed}, the assumptions of Theorem \ref{theo:nagel-implies-cassini} and 
	Theorem \ref{theo:no-restriction} are satisfied and these two results yield the claim.
\end{proof}

%% file: scalar_matrices_v7.tex
\section{Scalar Matrices}

For a scalar matrix $\A\in\C^{n\times n}$ the assumptions of Theorem \ref{theo:nagel-implies-gershgorin} 
-- and in particular condition~\eqref{restriction} -- 
are always fulfilled. 
Furthermore the set $S_{kj}(\A)$ can be simplified to
\begin{align*}
	S_{kj}(\A)\equiv \big\{\lambda\in\C\ \big\vert\ 
		\hspace{-.35cm}	\sum_{\quad i=1, i\ne k}^n |a_{ik}a_{kj}+(1-\delta_{ij})(\lambda-a_{kk})a_{ij}|\geq
		|\lambda-a_{jj}||\lambda-a_{kk}|\big\}\ .
\end{align*}
We know from Theorem~\ref{theo:nagel-implies-cassini} that our Schur sets are included in the Cassini ovals. 
The following example shows that the estimate derived in the previous section is \emph{strictly} better than 
that based on the Cassini ovals. 
On the other hand, the method based on Cassini ovals is computationally less intensive, see Remark \ref{rem:nn-1} below.

\begin{exa} Consider the matrix
	\begin{align*}
			\A=\begin{pmatrix}1 & 1 &1\\
						1 & 1 & 1\\
						1 & 1 & 1
					\end{pmatrix}
	\end{align*} whose eigenvalues are notoriously $0$ and $3$. 
	With the Schur sets we get the inclusion
		$$ \sigma(\A) \subset [0,3] \, ,$$
	according to
		$$\mathcal{R}_{kj}(\lambda)=|(\lambda-1)^{-2}+(\lambda-1)^{-1}|+|\lambda-1|^{-2}\geq1 \Leftrightarrow |\lambda|+1\geq|\lambda-1|^2,$$
	for $k\not=j\in\{1,2,3\}$.
	Instead, the Cassini method gives the inclusion
		\[
			|\lambda-1|^2\leq 1
		\]
	and therefore 
		$$\sigma(\mathbb A)\subset [-1,3]\ .$$
	The same estimate is yielded by Gershgorin's method.
\end{exa}

\begin{exa}
	Consider the matrix 
	\begin{align*}
			\A=\begin{pmatrix}2.3 &-1.6 &-0.8 &1\\
						-1.6 &3.3 &-0.7 &0.8\\
						-0.8 &-0.7 &1.1 &-0.3\\
						1 &0.8 &-0.3 &8.1 \end{pmatrix} \ .
	\end{align*}
	Even though the eigenvalues can be computed by hand, the explicit expressions are quite lengthy 
	and so we did the following computation using \emph{MATLAB R2009b}. 
	The eigenvalues of $\A$ are $-0.01..$, $1.97..$, $4.47..$ and  $8.36..$. 
	As the matrix $\A$ is hermitian, we know $\sigma(\A)\subset\R$. 
	This observation together with Corollary \ref{cor:nagel-implies-gershgorin} yields the inclusion
		\[
			\sigma(\A)\subset\lbrack-0.33\ldots,4.53\ldots\rbrack\cup\lbrack7.45\ldots,8.40\ldots\rbrack\  .
		\] 
	This example is interesting since while we get two disjoint intervals, 
	Ostrowski's method based on the Cassini ovals imply the localization 
	\[
		\sigma(\A)\subset\lbrack-0.84\ldots,9.20\ldots\rbrack\ .
	\]
\end{exa}

\begin{rema}\label{rem:nn-1}
	For an $(n \times n)$ scalar matrix, in order to determine the set $S_{k}(\A)$, 
	we have to solve the inequalities $\mathcal R_{kj}\ge 1$, $j \ne k$ and take the union of the sets of solutions. 
	By Corollary~\ref{cor:nagel-implies-gershgorin} we have then to take the intersection of the $n$ sets $S_k(\A)$. 
	Thus, our methods allows for a localization of the spectrum of $\A$ by solving a total of $n(n-1)$ inequalities. 
	Admittedly, Cassini is computationally less expensive: 
	In order to determine the set $C(\A)$, one has to solve $\frac12 n(n-1)$ inequalities.
\end{rema}